\documentclass[graybox]{svmult}

\usepackage{type1cm}        
\usepackage{makeidx}         
\usepackage{graphicx}        
\usepackage{stmaryrd}
\usepackage{multicol}        
\usepackage[bottom]{footmisc}

\usepackage{relsize}
\usepackage{newtxtext}       %
\usepackage[varvw]{newtxmath}       
\usepackage{tikz} 
\usepackage{tikz-cd}
\usepackage{mathtools}
\usepackage{float}
\usepackage{tabularx}
\usepackage{booktabs}
\usetikzlibrary{positioning}

\makeindex             

\begin{document}

\title*{Ternary associator, ternary commutator and ternary Lie algebra at cube roots of unity}
\author{Viktor Abramov\orcidID{0000-0001-7174-8030} 
}
\authorrunning{Ternary Lie algebra}
\institute{Viktor Abramov \at University of Tartu, Narva mnt 18, 51009 Tartu, Estonia  \email{viktor.abramov@ut.ee}
}
%
%
\maketitle

\abstract*{We extend the concepts of the associator and commutator from algebras with a binary multiplication law to algebras with a ternary multiplication law using cube roots of unity. By analogy with the Jacobi identity for the binary commutator, we derive an identity for the proposed ternary commutator. While the Jacobi identity is based on the cyclic permutation group of three elements $\mathbb Z_3$, the identity we establish for the ternary commutator is based on the general affine group $GA(1,5)$. We introduce the notion of a ternary Lie algebra at cube roots of unity. A broad class of such algebras is constructed using associative ternary multiplications of rectangular and cubic matrices. Furthermore, a complete classification of non-isomorphic two-dimensional ternary Lie algebras at cube roots of unity is obtained.}
\abstract{We extend the concepts of the associator and commutator from algebras with a binary multiplication law to algebras with a ternary multiplication law using cube roots of unity. By analogy with the Jacobi identity for the binary commutator, we derive an identity for the proposed ternary commutator. While the Jacobi identity is based on the cyclic permutation group of three elements $\mathbb Z_3$, the identity we establish for the ternary commutator is based on the general affine group $GA(1,5)$. We introduce the notion of a ternary Lie algebra at cube roots of unity. A broad class of such algebras is constructed using associative ternary multiplications of rectangular and cubic matrices. Furthermore, a complete classification of non-isomorphic two-dimensional ternary Lie algebras at cube roots of unity is obtained.}
\section{Introduction}
In this chapter, we consider the question of extending the concept of Lie algebras from algebras with a binary multiplication law to algebras with a ternary multiplication law. It is evident that approaches to such an extension of the Lie algebra concept may vary. The approach proposed by Filippov \cite{Filippov(1985)} is based on the observation that the Jacobi identity can be written in the form of a Leibniz formula, which shows that in the double Lie bracket, the outer bracket acts as a derivation of the inner one. This form of the Jacobi identity can be easily generalized to algebraic structures with ternary, or even more generally, $n$-ary multiplication laws. As a result, we obtain a structure currently known as an $n$-Lie algebra. It should be noted that a Lie algebra is a special case of an $n$-Lie algebra when $n=2$.

The concept of a Lie algebra is closely connected to the notion of a Poisson algebra. Recall that a Poisson algebra is a set with two algebraic structures, where the first is a commutative and associative algebra, and the second is a Lie algebra, with these two structures being compatible. Specifically, the Lie bracket acts as a derivation of commutative and associative multiplication. Therefore, it is not surprising that alongside the development of the theory of $n$-Lie algebras, Nambu proposed a generalization of Hamiltonian mechanics \cite{Nambu(1973)} based on an $n$-ary Poisson bracket satisfying all the requirements of the definition of an $n$-Lie algebra. In the late 1990s and the first decade of the 2000s, $n$-Lie algebras gained significant popularity as a research topic due to proposals to use them in the theory of M-branes \cite{Awata-Minich,Bagger-Lambert(2007),Bagger-Lambert(2008),Takhtajan}.

In this chapter, we introduce a novel approach to generalizing the concept of Lie algebras to algebraic structures equipped with a ternary multiplication law. This construction differs fundamentally from the Filippov–Nambu approach discussed earlier. Our methodology is motivated by and closely related to the ideas and frameworks emerging in the study of noncommutative geometry. Recall that within the framework of noncommutative geometry, various algebraic structures have been introduced and investigated, in which a parameter $q$, a primitive $N$-th root of unity, plays a central role. In particular, the notion of a graded $q$-differential algebra has been extensively introduced and studied in this context \cite{Abramov-Kerner,Dubois-Violette-Kerner,{Dubois-Violette-0}}. In such algebras, the differential $d$ satisfies the property $d^N=0$. When $N=2$, a graded $q$-differential algebra reduces to the well-known concept of a graded differential algebra.

The research presented in this chapter can be viewed as a continuation of the ideas and methods described in \cite{Abramov-Kerner-LeRoy}. The motivation for our proposed notion of a ternary Lie algebra at cube roots of unity stems from a ternary generalization of the Pauli exclusion principle and quantum nature of Lorentz invariance \cite{Abramov-Liivapuu,Kerner(2008),Kerner(2017),Kerner(2019)}. This generalization is based on the properties of the quark model.

Across our approach to a ternary extension of Lie algebras we employ the well known properties of cube roots of unity
\begin{equation}
1+\omega+\bar\omega=0,\;\;\omega^2=\bar\omega,\;\;\omega^3=\bar\omega^3=1,
\end{equation}
where $\bar\omega$ is the complex conjugate of a primitive cube root of unity $\omega$. 
Our approach is rooted in a key aspect of Lie algebra theory: the intrinsic connection between Lie algebras and associative algebras. Specifically, it is well known that any associative algebra can be transformed into a Lie algebra by defining a Lie bracket with the help of the commutator. Crucially, the Jacobi identity for this commutator is satisfied as a direct consequence of the associativity of the original algebra’s multiplication law. This foundational result plays a pivotal role in Lie algebra theory, as it facilitates the construction of a broad class of matrix Lie algebras. Matrix Lie algebras are of profound importance not only within the theoretical framework of Lie algebra theory but also in their numerous applications, particularly in mechanics, theoretical physics and differential geometry \cite{Helgason}.

Let us now examine in more detail the elementary observation mentioned earlier, namely that the Jacobi identity for the commutator holds due to the associativity of the multiplication law in an associative algebra. Assume that $x,y,z$ are elements of an algebra, which is not necessarily associative. If we expand the brackets of the commutator in the expression  (the left-hand side of the Jacobi identity)
\begin{equation}
[[x,y],z]+[[y,z],x]+[[z,x],y],
\label{intr: Jacobi identity}
\end{equation}
we obtain the expression consisting of 12 terms. Each term corresponds to a permutation of the three factors in the product $x\cdot y\cdot z$ (where $\cdot$ denotes the multiplication in our algebra) and appears twice: once in the form $(x\cdot y)\cdot z$ and once in the form $x\cdot (y\cdot z)$, with opposite signs. For example, the resulting expression will contain the terms 
\begin{equation}
(x\cdot y)\cdot z+(-1)\,x\cdot (y\cdot z)
\label{intr: associator}
\end{equation}
that is, the associator of three elements $x,y,z$. Thus, the entire expression will be a linear combination of six associators. If we now assume that an algebra is associative, all six associators will vanish, and we will obtain the Jacobi identity. Thus, we conclude that the Jacobi identity for the commutator 
\begin{equation}
[x,y]=x\cdot y +(-1)\,y\cdot x
\label{intr: binary commutator}
\end{equation} 
holds due to two key factors: first, the associativity of a multiplication, and second, the presence of the minus sign multiplying the second term on the right-hand side of the commutator.

These elementary considerations have been presented to clarify the analogy that will guide our approach to algebras with a ternary multiplication law. Let $\cal A$ be a vector space over the field of complex numbers $\mathbb C$, equipped with a trilinear mapping $(a,b,c)\in{\cal A}\times{\cal A}\times{\cal A}\to a\cdot b\cdot c\in{\cal A}$, which we refer to as ternary multiplication. In what follows a vector space $\cal A$ endowed with a ternary multiplication will be referred to as a ternary algebra. To begin, recall that in the case of ternary multiplication, associativity can be expressed in one of two forms \cite{Carlsson}:
\begin{equation}
(a \cdot b \cdot c) \cdot g \cdot h = a \cdot (b \cdot c \cdot g) \cdot h = a \cdot b \cdot (c \cdot g \cdot h),
\label{intr: associativity I}
\end{equation}
or
\begin{equation}
(a \cdot b \cdot c) \cdot g \cdot h = a \cdot (g \cdot c \cdot b) \cdot h = a \cdot b \cdot (c \cdot g \cdot h).
\label{intr: associativity II}
\end{equation}
where $a,b,c,g,h\in {\cal A}$. Associativity of the form \eqref{intr: associativity I} is referred to as associativity of the first kind, while associativity of the form \eqref{intr: associativity II} is called associativity of the second kind. In the case where the kind of associativity is not important, that is, when ternary multiplication is associative either of the first kind or the second, we will simply refer to it as a ternary associative algebra. It is worth noting that, in the context of ternary associativity, there is not yet a universally established terminology. The ternary associativity conditions \eqref{intr: associativity I} and \eqref{intr: associativity II} appear in various articles under different names. We will not enumerate these alternative terms, which can be found in \cite{Abramov,Zapata-Arsiwalla-Beynon(2024)}.

Now, we turn to the question of a ternary analog of the commutator \eqref{intr: binary commutator} and a ternary analog of the Jacobi identity. We aim to construct these two structures in the context of associative ternary multiplication. In our approach, the structure of the associator \eqref{intr: associator} plays a central role, as it unifies the concepts of the commutator, associativity and Jacobi identity into a cohesive framework. Formula \eqref{intr: associator}, in a figurative sense, serves as one of the "building blocks" through which we derive the Jacobi identity in the context of associative multiplication.

Thus, to generalize the concept of Lie algebras to structures with ternary multiplication laws, it is necessary to extend the notion of the associator \eqref{intr: associator} to ternary or, more generally, $n$-ary multiplication laws. To achieve this, we propose leveraging the primitive $n$-th roots of unity ($n\geq 2$) and their fundamental property
$$
1+q+q^2+\ldots+q^{n-1}=0,
$$
where $q$ is a primitive $n$-th root of unity. In the binary multiplication case, the primitive second-order root of unity is $-1$, and the associator has the form \eqref{intr: associator}.
In the case of ternary multiplication, according to the analogy mentioned above, we must use a primitive third-order root of unity, $\omega$. Since in this case there are two notions of associativity, we obtain two ternary associators. 
\begin{definition}
Let $\cal A$ be a ternary algebra whose ternary multiplication is denoted by $a\cdot b\cdot c$, where $a,b,c\in{\cal A}$. A linear combination of three ternary products
$$
(a \cdot b \cdot c) \cdot g \cdot h,\;\;\;
       a \cdot (b \cdot c \cdot g) \cdot h,\;\;
            a \cdot b \cdot (c \cdot g \cdot h).
$$
with coefficients $1,\omega,\bar\omega$ will be referred to as a ternary $\omega$-associator of the first kind. By this, we mean that the third-order roots of unity $1,\omega,\bar\omega$  can appear in any order within the linear combination, but all three must be used. Analogously a linear combination of three ternary products
$$
(a \cdot b \cdot c) \cdot g \cdot h,\;\;\;
       a \cdot (g \cdot c \cdot b) \cdot h,\;\;
            a \cdot b \cdot (c \cdot g \cdot h).
$$
with coefficients $1,\omega,\bar\omega$ will be referred to as a ternary $\omega$-associator of the second kind.
\end{definition}
It is evident that if a ternary multiplication satisfies the property of associativity of the first (second) kind, then the ternary $\omega$-associator of the first (second) kind vanishes identically due to the property $1+\omega+\bar\omega=0$. 

Note that in the structure we are developing, the primitive third-order roots of unity $\omega,\bar\omega$ are entirely equivalent. This leads to the fact that the structure of ternary $\omega$-associator exhibits a reflection-type symmetry $\omega\, \leftrightarrow\, \bar\omega$, meaning that replacing one root $\omega$ with another $\bar\omega$ does not fundamentally alter the defined structure. Thus, if we have a ternary $\omega$-associator of the first or second type, we can define a new ternary $\omega$-associator by replacing the primitive cube root of unity $\omega$ with the primitive cubic root of unity $\bar\omega$, and vice versa. It is worth noting that the sum of a ternary $\omega$-associator and its reflected counterpart results in a linear combination of ternary products with integer coefficients $-1$ and $+1$. Indeed, if the original ternary $\omega$-associator of the first kind has the form
$$
(a \cdot b \cdot c) \cdot g \cdot h + \bar\omega \; a \cdot (b \cdot c \cdot g) \cdot h + {\omega} \; a \cdot b \cdot (c \cdot g \cdot h),
$$
then its reflected version takes the form
$$
(a \cdot b \cdot c) \cdot g \cdot h + \omega \; a \cdot (b \cdot c \cdot g) \cdot h + \bar{\omega} \; a \cdot b \cdot (c \cdot g \cdot h).
$$
Adding these two linear combinations, we obtain
$$
(a \cdot b \cdot c) \cdot g \cdot h - a \cdot (b \cdot c \cdot g) \cdot h 
     +(a \cdot b \cdot c) \cdot g \cdot h- a \cdot b \cdot (c \cdot g \cdot h).
$$
It is worth to note that the first two terms in the obtained expression, so to speak, measure the partial associativity of a ternary multiplication when shifting round brackets from the leftmost position to the center, while the next two terms measure the partial associativity of a ternary multiplication when shifting the round brackets from the leftmost position to the rightmost.
\section{Ternary $\omega$-commutator}
The next step in the structure we are developing is to find a possible ternary analogue of the commutator, in which cube roots of unity would be used, and a ternary analogue of the Jacobi identity. It is natural to assume that the left-hand side of a ternary analogue of the Jacobi identity is given by a sum of double brackets of a ternary analogue of the commutator. Moreover, the following property should hold: expanding the double brackets of a ternary analogue of the commutator at the left-hand side of an identity yields an expression that is a linear combination of ternary $\omega$-associators. Consequently, in the case of ternary associativity, all ternary $\omega$-associators vanish, leading to the identity that we shall consider as the ternary analogue of the Jacobi identity.

We begin with a ternary analogue of the commutator. It should be noted immediately that the approach proposed in this paper differs from the existing approach, which is based on the skew-symmetry of the commutator. That is, an $n$-ary commutator or the alternating sum \cite{Bremner-Peresi(2006)} is skew-symmetric, that is, it remains unchanged under an even permutation of its arguments and changes sign under an odd permutation. For example, in the definition of an 
$n$-Lie algebra, it is required that an 
$n$-ary commutator be skew-symmetric and satisfy the Filippov–Jacobi identity. If $\cal A$ is a ternary algebra with a ternary product $a\cdot b\cdot c$, then one can consider the alternating sum (or ternary commutator) 
\begin{eqnarray}
&& a\cdot b\cdot c+b\cdot c\cdot a+c\cdot a\cdot b-c\cdot b\cdot a-b\cdot a\cdot c-a\cdot c\cdot b\label{intr: alternated ternary commutator}\\
&& = \underline{a}\cdot\underline{b}\cdot c 
       + \underline{b}\cdot c\cdot\underline{a} 
           +c\cdot\underline{a}\cdot\underline{b}
                                        \label{intr: form 2 of alternated ternary commutator}
\end{eqnarray}
where 
\begin{equation}
\underline{a}\cdot\underline{b}\cdot c=a\cdot b\cdot c-b\cdot a\cdot c,\;\;
    \underline{b}\cdot c\cdot\underline{a}=b\cdot c\cdot a-a\cdot c\cdot b,\;\;
      c\cdot\underline{a}\cdot\underline{b}=c\cdot a\cdot b-c\cdot b\cdot a.
\end{equation}
The expression \eqref{intr: alternated ternary commutator} is usually considered as a ternary analog of the commutator \cite{Bremner-Peresi(2006),Nambu(1973),Takhtajan}. The question of extending the concept of Lie algebras to algebras with an \(n\)-ary multiplication law via the alternating sum \eqref{intr: alternated ternary commutator} is studied in \cite{Bremner-Peresi(2006)}, where the minimal identities for the alternating sum \eqref{intr: alternated ternary commutator} are found in the case of (totally or partially) associative ternary products. It is worth noting that in \cite{Bremner-Peresi(2006)}, the ternary associativity of the second kind is not considered. The reason why the expression \eqref{intr: alternated ternary commutator} is considered as a ternary analogue of the commutator lies in its representation in the form \eqref{intr: form 2 of alternated ternary commutator}, where its connection to the binary commutator becomes apparent. In other words, it can be interpreted as a measure of the non-commutativity of a ternary multiplication with respect to pairs of its arguments.

It is evident that the approach to ternary analog of the commutator via the alternating sum \eqref{intr: alternated ternary commutator} is not consistent with our notion of ternary $\omega$-associator. In order for a ternary commutator, we are looking for, to be consistent with a ternary $\omega$-associator, we must use the cube roots of unity $1,\omega,\bar\omega$ as coefficients in an expression for ternary commutator. Since the cube roots of unity can be considered as the representation of the group of cyclic permutations of three elements, we define the ternary analogue of the commutator as a linear combination with coefficients 1, \(\omega\), \(\bar\omega\) of all permutations of the elements \(a\), \(b\), and \(c\) in the ternary product \(a \cdot b \cdot c\), where each cyclic permutation is multiplied by either \(\omega\) or \(\bar\omega\).
\begin{definition}
Let $\cal A$ be a ternary algebra with ternary product $a\cdot b\cdot c$. The ternary $\omega$-commutator is defined by the following formula
\begin{equation}
[a,b,c]=a\cdot b\cdot c+\omega\;b\cdot c\cdot a+\bar\omega\;c\cdot a\cdot b+
                   c\cdot b\cdot a+\bar\omega\;b\cdot a\cdot c+\omega\;a\cdot c\cdot b. 
\label{intr: ternary commutator}
\end{equation}
\end{definition}
Let us indicate the most important properties of the ternary $\omega$-commutator. As noted above, the structures we are developing exhibit reflection symmetry $\omega\leftrightarrow\bar\omega$. Thus, we can define a ternary $\omega$-commutator that will be a mirror image of commutator \eqref{intr: ternary commutator}, where by mirror image we mean the replacement $\omega\, \leftrightarrow\, \bar\omega$  in ternary $\omega$-commutator \eqref{intr: ternary commutator}. We shall refer to this ternary $\omega$-commutator as the reflected ternary $\omega$-commutator \eqref{intr: ternary commutator} and denote it as follows
\begin{equation}
[a,b,c]_{\omega\leftrightarrow\bar\omega}=a\cdot b\cdot c+\bar\omega\;b\cdot c\cdot a+\omega\;c\cdot a\cdot b+
                   c\cdot b\cdot a+\omega\;b\cdot a\cdot c+\bar\omega\;a\cdot c\cdot b. 
\label{intr: ternary commutator reflected}
\end{equation}
It is easy to verify that the permutation of the elements $a$ and $c$ in ternary $\omega$-commutator \eqref{intr: ternary commutator} is equivalent to transitioning from $\omega$-commutator \eqref{intr: ternary commutator} to its reflected counterpart \eqref{intr: ternary commutator reflected}. That is, we have the property
\begin{equation}
[c,b,a]=[a,b,c]_{\omega\leftrightarrow\bar\omega}.
\end{equation}
It is also easy to verify that a cyclic permutation of the arguments
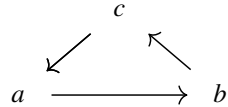
\begin{center}
    \begin{tikzcd}
    & c \arrow[dl, shorten <=5pt, shorten >=5pt] \arrow[dl, shorten <=5pt, shorten >=5pt] & \\
    a \arrow[rr, shorten <=5pt, shorten >=5pt] & & b \arrow[ul, shorten <=5pt, shorten >=5pt]
\end{tikzcd}
\end{center}
in the ternary $\omega$-commutator \eqref{intr: ternary commutator} and its reflected counterpart \eqref{intr: ternary commutator reflected} results in the appearance of the coefficients $\omega, \bar\omega$, respectively. That is, we have the following properties of the ternary $\omega$-commutator \eqref{intr: ternary commutator} and its reflected \eqref{intr: ternary commutator reflected} with respect to cyclic permutations of their arguments
\begin{eqnarray}
&& [a,b,c] = \omega\;[b,c,a],\label{omega-symmetry of commutator}\\
&& [a,b,c]_{\omega\leftrightarrow\bar\omega} = \bar\omega\;[b,c,a]_{\omega\leftrightarrow\bar\omega}.\label{omega-symmetry of reflected commutator}
\end{eqnarray}
Using the aforementioned symmetries of the ternary $\omega$-commutator with respect to permutations of its arguments, one can obtain important properties of the ternary $\omega$-commutator. First of all, it is easy to see that when all three arguments of the ternary $\omega$-commutator are equal, that is, $a=b=c$, the $\omega$-commutator identically vanishes $[a,a,a]=0$. However, when two of the arguments are equal, the ternary $\omega$-commutator does not vanish identically.
Indeed, if $a=c$, we obtain
\begin{eqnarray}
[a,b,a] &=& (a\cdot b\cdot a-b\cdot a\cdot a)+(a\cdot b\cdot a-a\cdot a\cdot b)\label{intr: noncommutativity 1}\\
         &=& \underline{a}\cdot\underline{b}\cdot a + a\cdot\underline{b}\cdot\underline{a}\label{intr: noncommutativity 0}\\
         &=& 2\;a\cdot b\cdot a-b\cdot a\cdot a-a\cdot a\cdot b.\label{intr: noncommutativity 2}
\end{eqnarray}
This formula is of interest in the sense that it indicates a connection between the ternary $\omega$-commutator \eqref{intr: ternary commutator} and the alternating sum \eqref{intr: alternated ternary commutator}. This becomes evident when comparing formulas \eqref{intr: form 2 of alternated ternary commutator} and \eqref{intr: noncommutativity 0}. As we have already noted, a pair of underlined symbols means that one has to apply alternation to this pair and this can be regarded as a "measurement" of the deviation from commutativity with respect to this pair of variables. Moreover, formula \eqref{intr: noncommutativity 2} shows that when two arguments in the ternary $\omega$-commutator \eqref{intr: ternary commutator} are equal, the linear combination of ternary products on the right-hand side of the ternary $\omega$-commutator has integer coefficients +1 and -1, bringing it closer to the commutator \eqref{intr: binary commutator} and the alternating sum \eqref{intr: alternated ternary commutator}. Figuratively speaking, one could say that the ternary $\omega$-commutator \eqref{intr: ternary commutator} makes a "jump" from integer coefficients in the case of binary commutator to the complex plane only when all three of its arguments are different.

An important characteristic of an algebra with binary multiplication is the existence of a unit element. This question plays an important role also in the case of ternary algebra. An analogue of a unit element in the case of a ternary algebra is the so-called biunit. Recall that an element $e$ of a ternary algebra $\cal A$ is called a right (left) biunit if, for any $a\in\cal A$ it satisfies the condition $a\cdot e\cdot e=a$ ($e\cdot e\cdot a=a$). An element $e$ is called a biunit if it is a right and a left biunit, and, additionally, it satisfies the condition $e\cdot a\cdot e=a$. Recall that in the case of a binary unital algebra (i.e., an algebra with binary multiplication and a unit element), the commutator of any element of this algebra with the unit element is zero. The ternary $\omega$-commutator \eqref{intr: ternary commutator} we propose has an analogous property, which means that if $e$ is a biunit of a ternary algebra $\cal A$, then for any element $a\in\cal A$ it holds $[e,a,e]=2\,e\cdot a\cdot e-e\cdot e\cdot a-a\cdot e\cdot e=0$. 
\section{General affine group of degree 1 and identity}
In the previous section, we defined the ternary $\omega$-commutator \eqref{intr: ternary commutator} and outlined its properties. An important question concerns an identity it satisfies. In this section, we will prove that if a ternary algebra is associative (of the first or second kind), then the ternary $\omega$-commutator \eqref{intr: ternary commutator} satisfies an identity based on the general affine group $GA(1,5)$. The right-hand side of this identity is zero, while the left-hand side is a sum of twenty double ternary $\omega$-commutators. To verify this identity, we apply formula \eqref{intr: ternary commutator} twice to each of the twenty double ternary $\omega$-commutators. As a result, all the terms obtained can be grouped into ternary $\omega$-associators either of the first or second kind, depending on the type of associativity. In sum, we obtain zero due to the associativity of ternary multiplication. Thus, we observe a complete analogy with the binary commutator, the associativity of binary algebra, and the Jacobi identity described in Introduction. This provides a basis for considering the obtained identity as a ternary analogue of the Jacobi identity.

The general affine group of degree one $GA(1,5)$ is the group of affine transformations over the field $\mathbb F_5$, that is, the group of transformations of the form $x\to a\,x+b$, where the operation of multiplication is the composition of transformations and $a\in \mathbb \mathbb F_5^\times=F_5\backslash\{0\}, b\in \mathbb F_5$. This group is non-abelian and its order is 20. The general affine group $GA(1,5)$ is the semidirect product of the multiplicative group $\mathbb F_5^\times$ and the additive group $\mathbb F_5$, that is, $GA(1,5)=\mathbb F_5^\times\rtimes \mathbb F_5$ with the multiplication defined by $(a_1,b_1)\cdot(a_2,b_2)=(a_1\,a_2, b_1+a_1\,b_2)$.

The general affine group of degree one $GA(1,5)$ is also a subgroup of the symmetric group on five elements $S_5$. In this paper, we will use this representation. The general affine group is generated by two cycles $\sigma=(1\; 2\; 3\; 4\; 5),\;\;\tau=(2\;4\;5\;3)$.
Then, we have the following:
$$
GA(1,5)=<\sigma,\tau\;|\; \sigma^5=e, \tau^4=e, \tau\,\sigma\,\tau^{-1}=\sigma^2>,
$$
where $e$ is the identity element of the group $GA(1,5).$ It follows that the general affine group of degree one $GA(1,5)$ contains two cyclic subgroups, $N$ and $H$ generated by $\sigma$ and $\tau$, respectively. Thus $N=\{e,\sigma,\sigma^2,\sigma^3,\sigma^4\}$ and $H=\{e,\tau,\tau^2,\tau^3\}$. The subgroup $N$ is a normal divisor of the general affine group of degree one. Notably, the general affine group of degree one is the semidirect product of the subgroups $N$ and $H$.

Consider the general affine group of degree one $GA(1,5)$ as a subgroup of the symmetric group $S_5$. Let $a_1,a_2,a_3,a_4,a_5$ be elements of a ternary algebra $\cal A$. Form the double ternary $\omega$-commutator
\begin{equation}
(a_1,a_2,a_3,a_4,a_5)\to\big[[a_1,a_2,a_3],a_4,a_5\big].
\label{sec1: double ternary commutator}
\end{equation}
Define the action of a permutation $\rho\in GA(1,5)$ by the following formula
$$
\rho\cdot\big[[a_{i_1},a_{i_2},a_{i_3}],a_{i_4},a_{i_5}\big] = \big[[a_{i_{\rho(1)}},a_{i_{\rho(2)}},a_{i_{\rho(3)}}],a_{i_{\rho(4)}},a_{i_{\rho(5)}}\big],
$$
where $(i_1,i_2,i_3,i_4,i_5)$ is a permutation of integers $(1,2,3,4,5)$.
Thus, each permutation $\rho$ of the general affine group $GA(1,5)$ induces the permutation of the elements in the double ternary $\omega$-commutator \eqref{sec1: double ternary commutator}. Let us define the sum of double ternary $\omega$-commutators
\begin{equation}
\Omega\,(a_1,a_2,a_3,a_4,a_5)=\sum_{\rho\in GA(1,5)}\;\rho\cdot\big[[a_1,a_2,a_3],a_4,a_5\big].
\label{sec1: Omega polynomial}
\end{equation}
It is evident that there are twenty double ternary $\omega$-commutators in \eqref{sec1: Omega polynomial}, where each is a double ternary $\omega$-commutator of the form 
$\big[[a_{\rho(1)},a_{\rho(2)},a_{\rho(3)}],a_{\rho(4)},a_{\rho(5)}\big]$, where $\rho$ is a permutation from the general affine group $GA(1,5)$.
To provide a more precise description of the structure of $\Omega\,(a_1,a_2,a_3,a_4,a_5)$ , we list all elements of the general affine group $GA(1,5)$, arranging them in the following table
\begin{eqnarray}
&& e,\; \sigma,\; \sigma^2,\;\sigma^3,\;\sigma^4,\label{first 5}\\
&& \tau,\;\sigma\tau,\;\sigma^2\tau,\;\sigma^3\tau,\;\sigma^4\tau,\label{second 5}\\
&& \tau^2,\;\sigma\tau^2,\;\sigma^2\tau^2,\;\sigma^3\tau^2,\;\sigma^4\tau^2,\label{third 5}\\
&& \tau^3,\;\sigma\tau^3,\;\sigma^2\tau^3,\;\sigma^3\tau^3,\;\sigma^4\tau^3.\label{fourth 5}
\end{eqnarray}
The permutations in the first row of the table correspond to cyclic permutations of the elements in the double ternary $\omega$-commutator \eqref{sec1: double ternary commutator}. The first permutation in the second row produces the double ternary commutator $\big[[a_1,a_4,a_2],a_5,a_3\big]$, and the subsequent permutations in this row perform cyclic permutations of the elements in this double $\omega$-commutator. The following rows exhibit a similar structure. Thus, by introducing the notation for cyclic permutations using the formula
$$
\mathlarger{\circlearrowleft}\big[[a_{i_1},a_{i_2},a_{i_3}],a_{i_4},a_{i_5}\big]=\sum_{k=0}^4\;\sigma^k\cdot \big[[a_{i_1},a_{i_2},a_{i_3}],a_{i_4},a_{i_5}\big],
$$
we can express $\Omega\,(a_1,a_2,a_3,a_4,a_5)$  in the form
\begin{eqnarray}
\Omega(a_1,a_2,a_3,a_4,a_5)=&&\!\!\!\mathlarger{\circlearrowleft}\Big(\big[[a_1,a_2,a_3],a_4,a_5\big] +\big[[a_1,a_4,a_2],a_5,a_3\big]\nonumber\\
      &&\;\; +\big[[a_1,a_5,a_4],a_3,a_2\big]+\big[[a_1,a_3,a_5],a_2,a_4\big]\Big).
\label{sec1: structure of Omega}
\end{eqnarray}
In formula \eqref{sec1: structure of Omega}, the symbol for cyclic permutations is applied to each of the four terms enclosed in round brackets.
\begin{theorem}
Let $\cal A$ be an associative ternary algebra. Then for any five elements $a_1,a_2,a_3,a_4,a_5\in\cal A$ we have $\Omega(a_1,a_2,a_3,a_4,a_5)=0$. 
\label{theorem: identity}
\end{theorem}
\begin{proof}
The theorem states that the vanishing of the polynomial \eqref{sec1: structure of Omega} is independent of the kind of ternary associativity. In the proof, we assume that the ternary multiplication is associative of the second kind. The case of associativity of the first kind is treated similarly, with minor modifications.

First, note that the inner ternary $\omega$-commutator in each double ternary $\omega$-commutator of \eqref{sec1: structure of Omega} contains three different elements of the set 
$$
{\cal S}=\{a_1,a_2,a_3,a_4,a_5\}.
$$ 
If we disregard the order of these three elements, we obtain 10 such triples, that is, $C^5_3=10$. However, in \eqref{sec1: structure of Omega}, there are twenty double ternary $\omega$-commutators. Consequently, all double ternary $\omega$-commutators can be grouped into pairs, where each pair consists of two double ternary $\omega$-commutators containing in the inner ternary $\omega$-commutator the same triple of elements (but arranged differently) from the set $\cal S$. Let 
$$
{\mathfrak t}=\big[[a_{i_1},a_{i_2},a_{i_3}],a_{i_4},a_{i_5}\big]
$$ 
be a double ternary $\omega$ -commutator from the sum \eqref{sec1: structure of Omega}. By listing all the double ternary $\omega$-commutators of \eqref{sec1: structure of Omega}, it is easy to see that the double ternary $\omega$-commutator 
$$
\tilde{\mathfrak t}=\big[[a_{i_3},a_{i_2},a_{i_1}],a_{i_5},a_{i_4}\big]
$$ 
also belongs to \eqref{sec1: structure of Omega}. These two double ternary $\omega$-commutators ${\mathfrak t},\tilde{\mathfrak t}$ contain the product $(a_{i_1}\cdot a_{i_2}\cdot a_{i_3})\cdot a_{i_4}\cdot a_{i_5}$ with the coefficients $1$ and $\omega$, respectively. Parentheses indicate the order of ternary multiplications, and the coefficients are easily determined from the symmetry properties of the ternary $\omega$-commutator.  

Now, any cyclic permutation of arguments in the double ternary $\omega$-commutator $\mathfrak t$ gives the double ternary $\omega$-commutator that also belongs to \eqref{sec1: structure of Omega}. Consequently, the double ternary $\omega$-commutators
$$
{\mathfrak s}=\big[[a_{i_2},a_{i_3},a_{i_4}],a_{i_5},a_{i_1}\big],\;\;
         \tilde{\mathfrak s}=\big[[a_{i_4},a_{i_3},a_{i_2}],a_{i_1},a_{i_5}\big],
$$
are the terms of the sum \eqref{sec1: structure of Omega}. It easy to find that the double ternary $\omega$-commutators ${\mathfrak s},\tilde{\mathfrak s}$ contain the product $a_{i_1}\cdot (a_{i_4}\cdot a_{i_3}\cdot a_{i_2})\cdot a_{i_5}$ with the coefficient $\bar\omega$.  

Finally, performing the cyclic permutation in the double ternary $\omega$-commutator ${\mathfrak s}$, we obtain two more double ternary commutators 
$$
{\mathfrak v}=\big[[a_{i_3},a_{i_4},a_{i_5}],a_{i_1},a_{i_2}\big],\;\;\;  
   \tilde{\mathfrak v}=\big[[a_{i_5},a_{i_4},a_{i_3}],a_{i_2},a_{i_1}\big],
$$ 
which belong to \eqref{sec1: structure of Omega}. Repeating the previous reasoning, we conclude that double ternary $\omega$-commutators ${\mathfrak v},\tilde{\mathfrak v}$ contain the ternary product $a_{i_1}\cdot a_{i_2}\cdot (a_{i_3}\cdot a_{i_4}\cdot a_{i_5})$ with coefficients $\omega$ and $1$, respectively. The obtained results can be represented in the form of following tables, where the first row consists of double ternary $\omega$-commutators, and the second row below them shows the ternary products they contain:
\vskip.3cm
\begin{tabular}{|c|c|c|}
\hline
  \rule{0pt}{2.5ex} 
  $\big[[a_{i_1},a_{i_2},a_{i_3}],a_{i_4},a_{i_5}\big]$ & $\big[[a_{i_2},a_{i_3},a_{i_4}],a_{i_5},a_{i_1}\big]$ & $\big[[a_{i_3},a_{i_4},a_{i_5}],a_{i_1},a_{i_2}\big]$\\
  \rule{0pt}{2.8ex} 
  $(a_{i_1}\cdot a_{i_2}\cdot a_{i_3})\cdot a_{i_4}\cdot a_{i_5}$ & $\bar\omega\;a_{i_1}\cdot (a_{i_4}\cdot a_{i_3}\cdot a_{i_2})\cdot a_{i_5}$ & $\omega\;a_{i_1}\cdot a_{i_2}\cdot (a_{i_3})\cdot a_{i_4}\cdot a_{i_5})$ \\
  \hline
\end{tabular}
\vskip.3cm
\noindent
and
\vskip.3cm
\begin{tabular}{|c|c|c|}
\hline
  \rule{0pt}{2.5ex} 
  $\big[[a_{i_3},a_{i_2},a_{i_1}],a_{i_5},a_{i_4}\big]$ & $\big[[a_{i_4},a_{i_3},a_{i_2}],a_{i_1},a_{i_5}\big]$ & $\big[[a_{i_5},a_{i_4},a_{i_3}],a_{i_2},a_{i_1}\big]$\\
  \rule{0pt}{2.5ex} 
  $\omega\;(a_{i_1}\cdot a_{i_2}\cdot a_{i_3})\cdot a_{i_4}\cdot a_{i_5}$ & $\bar\omega\;a_{i_1}\cdot (a_{i_4}\cdot a_{i_3}\cdot a_{i_2})\cdot a_{i_5}$ & $a_{i_1}\cdot a_{i_2}\cdot (a_{i_3}\cdot a_{i_4}\cdot a_{i_5})$ \\
  \hline
\end{tabular}
\vskip.3cm
\noindent
It is easy to see that the given tables list all double ternary $\omega$-commutators from the sum \eqref{sec1: structure of Omega} that contain one of the three ternary products
\begin{equation}
(a_{i_1}\cdot a_{i_2}\cdot a_{i_3})\cdot a_{i_4}\cdot a_{i_5},\;\;
                        a_{i_1}\cdot (a_{i_4}\cdot a_{i_3}\cdot a_{i_2})\cdot a_{i_5},\;\;
                              a_{i_1}\cdot a_{i_2}\cdot (a_{i_3}\cdot a_{i_4}\cdot a_{i_5}).
\label{sec1: double ternary products}
\end{equation}
\noindent
We observe that the sum of the ternary products appearing in the second row of the first table forms the ternary $\omega$-associator. Similarly, the sum of the ternary products in the second row of the second table also constitutes the ternary $\omega$-associator. Due to ternary associativity of the second kind, these ternary $\omega$-associators vanish.  

Thus, we have proved that all ternary products of the form \eqref{sec1: double ternary products}, where the permutation $(i_1,i_2,i_3,i_4,i_5)$ belongs to the general affine group, can be combined into ternary $\omega$-associators and, by ternary associativity, vanish. Note that the number of such permutations is 20, which coincides with the number of elements in the general affine group.  

Now, assume that we apply the same cyclic permutation of three elements in each inner ternary $\omega$-commutator in \eqref{sec1: structure of Omega}. Then each ternary $\omega$-commutator in \eqref{sec1: structure of Omega} will be multiplied by the same scalar, that is, either all terms in \eqref{sec1: structure of Omega} will be multiplied by $\omega$ or by $\bar\omega$, depending on the cyclic permutation we choose. Factoring out this common coefficient, we obtain the sum of double commutators analogous to \eqref{sec1: structure of Omega}, and we can apply to this sum our previous arguments. That is, the ternary products can once again be grouped into ternary $\omega$-associators, and by associativity, their sum vanishes. Since we can perform two cyclic permutations, we prove the vanishing of an additional 40 permutations of the form \eqref{sec1: double ternary products}, leading to a total of 60.  

The remaining 60 permutations correspond to non-cyclic permutations of triples of elements in the inner ternary $\omega$-commutators of \eqref{sec1: structure of Omega}. It is straightforward to verify that the properties of the ternary $\omega$-commutator (where a non-cyclic permutation corresponds to the transition to a reflected ternary $\omega$-commutator) reduce this case to the one considered above. Consequently, we ultimately conclude that all ternary products of the form \eqref{sec1: double ternary products}, where $(i_1,i_2,i_3,i_4,i_5)$ is an arbitrary permutation from $S_5$, combine into ternary $\omega$-associators, and by ternary associativity, the entire expression vanishes, that is, $\Omega(a_1,a_2,a_3,a_4,a_5)=0$.
\end{proof}
\section{Ternary Lie algebra at cube roots of unity}
In this section, we define the notion of ternary Lie algebra at cube roots of unity. At the core of a ternary Lie algebra at cube roots of unity lies a ternary Lie bracket, which satisfies the symmetries of the ternary $\omega$-commutator \eqref{omega-symmetry of commutator} and the identity stated in Theorem \ref{theorem: identity}. Theorem \ref{theorem: identity} allows us to construct a broad class of ternary Lie algebras at cube roots of unity. Indeed, given a ternary associative algebra over the field of complex numbers, we can equip it with the ternary $\omega$-commutator, thereby obtaining a ternary Lie algebra at cube roots of unity. We also give the definition of simple ternary Lie algebra at cube roots of unity. Then we introduce the structure constants of a ternary Lie algebra at cube roots of unity and derive a system of equations they must satisfy.
\begin{definition}
Let $\cal T$ be a vector space over the field of complex numbers and $\omega,\bar\omega$ be primitive third-order roots of unity. Then, $\cal T$ is said to be a {ternary Lie algebra at cube roots of unity} if $\cal T$ is endowed with a trilinear mapping $[-,-,-]:{\cal T}\times{\cal T}\times{\cal T}\to {\cal T}$ such that for any $a,b,c,f,g\in{\cal T}$ the following conditions are satisfied:
\begin{itemize}
\item
$[a,b,c]=\omega\;[b,c,a]=\overline\omega\;[c,a,b],$
\item 
$\circlearrowleft \Big(\big[[a,b,c],f,g\big]+\big[[a,f,b],g,c\big]
+\big[[a,g,f],c,b\big]+\big[[a,c,g],b,f\big]\Big) =0.$
\end{itemize}
\label{sec1: Definition of ternary Lie algebra}
\end{definition}
In order to simplify presentation of the defined structure, we will use the following terminology:  
\begin{enumerate}
\item The trilinear mapping $[-,-,-]:{\cal T}\times{\cal T}\times{\cal T}\to {\cal T}$ will be referred to as the ternary Lie bracket.  
\item The property of the ternary Lie bracket stated in the first condition of Definition \ref{sec1: Definition of ternary Lie algebra} will be called the $\omega$-symmetry of the ternary Lie bracket with respect to cyclic permutations.  
\item The identity given in the second condition of Definition \ref{sec1: Definition of ternary Lie algebra} will be referred to as the $GA(1,5)$-identity.  
\item A ternary Lie algebra at the cube roots of unity will be called a ternary $\omega$-Lie algebra.
\end{enumerate}
\begin{definition}
Let $\cal T$ be a ternary $\omega$-Lie algebra and $\cal I\subset {\cal T}$ be its subspace. Then, $\cal I$ is said to be an ideal of a ternary $\omega$-Lie algebra $\cal T$ if, for any $a\in {\cal I}$ and $x,y\in{\cal L}$, it holds $[a,x,y]\in{\cal I}.$ A ternary $\omega$-Lie algebra is said to be simple if it has no non-trivial ideals, that is, it has no ideals other than $\{0\}$ and $\cal L$.
\end{definition}
Let $\cal T$ be a ternary $\omega$-Lie algebra, where $\cal T$ is an $n$-dimensional vector space, and $e_1,e_2,\ldots,e_n$ be a basis for this vector space. We introduce the structure constants $C^m_{ikl}$ of a ternary $\omega$-Lie algebra $\cal T$ as follows:
\begin{equation}
[e_i,e_k,e_l]=C_{ikl}^m\;e_m,
\label{structure constants}
\end{equation}
where all indices take integer values from 1 to \( n \) and we use the Einstein convention of summation over repeated indices. When replacing one basis of a vector space $\cal T$ with another, the structure constants transform according to the tensor law, that is, they transform as a tensor with one contravariant and three covariant indices. This tensor has the $\omega$-symmetry with respect to the cyclic permutations of its covariant subscripts, that is,
\begin{equation}
C^m_{ikl}=\omega\;C^m_{kli}=\overline\omega\;C^m_{lik}.
\label{structure constants omega-symmetry}
\end{equation}
It is evident that the $GA(1,5)$-identity imposes additional constraints on the structure constants of $\cal T$, which can be expressed as a system of equations. To present this system in a more compact form, we will use, as before, the symbol for cyclic permutations of five elements, in this case, the indices. The five indices that undergo cyclic permutations will be underlined. Thus, we define
$$
\circlearrowleft C^m_{\underline{i}\,\underline{k}\,\underline{l}}\,C^p_{m\,\underline{r}\,\underline{s}}=C^m_{ikl}\,C^p_{mrs}+C^m_{klr}\,C^p_{msi}+C^m_{lrs}\,C^p_{mik}+C^m_{rsi}\,C^p_{mkl}+C^m_{sik}\,C^p_{mlr}.
$$
It follows from the $GA(1,5)$-identity that the structure constants of a ternary $\omega$-Lie algebra $\cal T$ satisfy the system of {equations}, as follows: 
\begin{equation}
\circlearrowleft (C^m_{\underline{i}\,\underline{k}\,\underline{l}}\,C^p_{m\,\underline{r}\,\underline{s}}+C^m_{\underline{i}\,\underline{r}\,\underline{k}}\,C^p_{m\,\underline{s}\,\underline{l}}+C^m_{\underline{i}\,\underline{s}\,\underline{r}}\,C^p_{m\,\underline{l}\,\underline{k}}+C^m_{\underline{i}\,\underline{l}\,\underline{s}}\,C^p_{m\,\underline{k}\,\underline{r}})=0.
\label{identity for structure constants}
\end{equation}
\section{Associative ternary algebras of rectangular and cubic matrices}
In this section, we provide examples of associative ternary algebras. We describe a fairly general structure that enables the construction of associative ternary algebras of rectangular and cubic matrices.

First of all, we recall that associative ternary algebras are closely related to a more general algebraic structure known as a semiheap \cite{Hawthorn-Stokes,Kolar,Rybolowicz-Zapata,Wagner(1951),Wagner(1953)}. A set \(\mathfrak S \) is called a semiheap if it is equipped with a ternary multiplication that satisfies associativity of the second kind. Thus, we see that \( \mathfrak S \) forms a ternary algebra with associativity of the second kind whenever it possesses a vector space structure and the ternary multiplication is linear in each argument. The theory of semiheaps was developed by Wagner within the algebraic framework of the theory of manifolds in differential geometry. 

We now describe a general structure that enables the construction of associative ternary algebras. Let \( ({\cal R},\ast) \) be an associative (binary) algebra, and let \( \cal M \) be a right \({\cal R}\)-module. Let \( \beta:{\cal M}\times{\cal M}\to{\cal R} \) be a bilinear form on the module \( {\cal M} \). Then we can define ternary multiplication on the vector space $\cal M$ by the formula
\begin{equation}
x\diamond y\diamond z=x\cdot \beta(y,z).
\label{sec4: ternary multiplication beta}
\end{equation}
Analogously, if $\cal M$ is a left $\cal L$-module, where $(\cal L,\star)$ is an associative (binary) algebra, and $\alpha:{\cal M}\times{\cal M}\to {\cal L}$ is a bilinear form then we can define the ternary multiplication on the vector space $\cal M$ by the formula
\begin{equation}
x\odot y\odot z=\alpha(x,y)\cdot z.
\label{sec4: ternary multiplication alpha}
\end{equation}
\begin{proposition}
The ternary multiplication \eqref{sec4: ternary multiplication beta} is associative of the second kind if $\beta$ satisfies the equation
\begin{eqnarray}
\beta\big(u\cdot\beta(z,y),v\big)=\beta\big(y,z\cdot\beta(u,v)\big)=\beta(y,z)\ast\beta(u,v).
\label{sec4: condition for beta}
\end{eqnarray}
The ternary multiplication \eqref{sec4: ternary multiplication alpha} is associative of the second kind if $\alpha$ satisfies the equation
\begin{eqnarray}
&& \alpha\big(\alpha(x,y)\cdot z,u\big)=\alpha\big(x,\alpha(u,z)\cdot y\big)=\alpha(x,y)\star\alpha(z,u).
\label{sec4: condition for alpha}
\end{eqnarray}
\end{proposition}
\begin{proof}
We will prove second-kind associativity only for the ternary multiplication \eqref{sec4: ternary multiplication beta}, as the proof for the ternary multiplication \eqref{sec4: ternary multiplication alpha} follows analogously. We have
\begin{eqnarray}
&& (x\diamond y\diamond z)\diamond u\diamond v=(x\diamond y\diamond z)\cdot\beta(u,v)=
                                      x\cdot(\beta(y,z)\ast \beta(u,v)),\nonumber\\
&& x\!\diamond\! (u\diamond z\diamond y)\diamond v=x\!\cdot\!\beta(u\diamond z\diamond y,v)=
                                x\!\cdot\! \beta(u\!\cdot\!\beta(z,y),v)=x\!\cdot\!(\beta(y,z)\ast\beta(u,v)).\nonumber\\
&&  x\!\diamond\! y\diamond (z\diamond u\diamond v)=x\!\cdot\!\beta(y,z\diamond u\diamond v)=
                                        x\!\cdot\! \beta(y,z\cdot \beta(u,v)=x\!\cdot\!(\beta(y,z)\ast\beta(u,v)).\nonumber
\end{eqnarray}
\end{proof}
\noindent
It should be noted that if \( \cal M \) is a $({\cal L},{\cal R})$-bimodule and the bilinear forms $\alpha,\beta$ satisfy the equation $\alpha(x,y)\cdot z=x\cdot\beta(y,z)$, then the ternary products \eqref{sec4: ternary multiplication beta} and \eqref{sec4: ternary multiplication alpha} coincide.

We can use the proved statement to construct specific ternary associative algebras with second-kind associativity. First, we apply the described structure to rectangular matrices. Let \( {\mathfrak M}_{m,n} \) be the vector space of rectangular $m\times n$-matrices. This space is a right \({\mathfrak A}_n\)-module, where \( {\mathfrak A}_n \) is the algebra of square matrices of order \( n \). We define a bilinear form $\beta:{\mathfrak M}_{m,n}\times {\mathfrak M}_{m,n}\to {\mathfrak A}_n$ by \( \beta(A,B)=A^T\,B \), where $A,B$ are rectangular $m\times n$-matrices, $A^T$ stands for transposed matrix and $A^T\,B$ is a product of two matrices. It is straightforward to verify that \( \beta \) satisfies equation \eqref{sec4: condition for beta}. Consequently, the ternary multiplication on $(m,n)$-matrices $A\cdot B\cdot C=A\,B^T\,C$ is second-kind associative, and we obtain a ternary algebra with second-kind associativity, whose elements are $(m,n)$-matrices. If we consider the vector space \( {\mathfrak M}_{m,n} \) as a left ${\mathfrak A}_m$-module over the algebra ${\mathfrak A}_m$ of square matrices of order \( m \) and define the bilinear form by $\alpha(A,B)=A\,B^T$, we will evidently arrive at the same ternary algebra with second-kind associativity. If we now consider the ternary $\omega$-commutator on a ternary associative (second-kind) algebra of rectangular matrices, we obtain a ternary $\omega$-Lie algebra.

Let us consider a special case of the ternary $\omega$-Lie algebra $\mathfrak M_{m,n}$ constructed by means of rectangular $(m\times n)$-matrices, when $m=1$. Thus, the elements of the ternary $\omega$-Lie algebra are matrices consisting of a single row. These matrices can be identified with \( n \)-dimensional complex vectors. Consequently, the vector space of the ternary $\omega$-Lie algebra is an \( n \)-dimensional complex space $\mathbb C^n$. In this particular case we have $\alpha(x,y)=x\,y^T$, where $x,y\in{\mathbb C}^n$. In this case, the form $\alpha$ is symmetric, which simplifies the expression for the ternary $\omega$-commutator. Indeed, we have
\begin{eqnarray}
&& \alpha(x,y)\, z + \omega\,\alpha(y,z)\, x+\bar\omega\,\alpha(z,x)\, y+\alpha(z,y)\, x+\bar\omega\,\alpha(y,x)\, z+\omega\,\alpha(x,z)\, y\nonumber\\
        &&\qquad\qquad=   -\omega\,(\alpha(x,y)\, z+\omega\,\alpha(y,z)\, x+\bar\omega\,\alpha(z,x)\,y).
\end{eqnarray}
Thus, by discarding the irrelevant factor $-1$, we can define the ternary $\omega$-commutator in this particular case with a more concise expression, which is a linear combination of cyclic permutations, that is,
\begin{eqnarray}
&& [x,y,z]=\alpha(z,x)\, y+\omega\,\alpha(x,y)\,z+\bar\omega\,\alpha(y,z)\,x\nonumber\\
    && \qquad\quad\;\; =z\,x^T\,y+\omega\;x\,y^T\,z+\overline\omega\;y\,z^T\,x.\nonumber
\label{ternary commutator for vectors}
\end{eqnarray}
In this particular case, we can easily compute the structure constants of the ternary $\omega$-Lie algebra. Indeed let $e_1,e_2,\ldots,e_n$ be the canonical basis for $\mathbb C^n$, that is, the $i$th coordinate of a vector $e_i$ is 1, all other coordinates are equal to zero. Then, the structure constants of this ternary $\omega$-Lie algebra are as follows:
\begin{equation}
C^m_{ijk}=\delta_{ki}\,\delta^m_j+\omega\;\delta_{ij}\;\delta^m_k+\overline\omega\;\delta_{jk}\,\delta^m_i.
\end{equation}
If we calculate the structure constants (\ref{ternary commutator for vectors}) for the simplest case of two-dimensional complex vector space $\mathbb C^2$, then we obtain the following ternary commutation relations:
\begin{equation}
[e_1,e_2,e_1]=e_2,\;\;[e_2,e_1,e_2]=e_1.
\label{sec4: commutation relations 1}
\end{equation}
We denote the two-dimensional ternary $\omega$-Lie algebra with commutation relations (\ref{sec4: commutation relations 1}) by ${\mathfrak M}_{1,2}$.

We now proceed to the construction of associative ternary algebras using cubic matrices \cite{Abramov-Kerner-Liivapuu-Shitov}. A cubic matrix of order \( n \) is understood as a complex-valued entity with three indices ${\tt X}=({\tt X}_{ijk})$, each ranging over integer values from 1 to \( n \). Such matrices are also referred to as cubic or three-dimensional matrices. Let us denote the vector space of cubic matrices of order $n$ by ${\cal C}_n$. Then the vector space of cubic matrices of order $n$ is a right ${\mathfrak A}_n$-module, where ${\mathfrak A}_n$ is the algebra of square matrices of order $n$, if we define the right action of the algebra of $n$th order square matrices on the vector space of $n$th order cubic matrices ${\cal C}_n\times {\mathfrak A}_n\to {\cal C}_n$ by
\begin{equation}
({\tt X}\cdot A)_{ijl}={\tt X}_{ijk}\,A_{kl}.
\end{equation}
To construct a second-kind associative ternary algebra of cubic matrices, we need a ${\mathfrak A}_n$-valued bilinear form defined on the space of cubic matrices ${\cal C}_n$. Moreover, this bilinear form must satisfy equation \eqref{sec4: condition for beta}. We consider two such forms, \( \beta \) and \( \gamma \), which we define as follows
\begin{equation}
\beta({\tt X},{\tt Y})_{pk}={\tt X}_{rsp}\,{\tt Y}_{srk},\;\;\;
                \gamma({\tt X},{\tt Y})_{pk}={\tt X}_{rsp}\,{\tt Y}_{rsk}     
\end{equation}
where ${\tt X},{\tt Y}$ are $n$th order cubic matrices. We will show that the bilinear form \( \beta \) satisfies equation \eqref{sec4: condition for beta}. The proof for the form \( \gamma \) follows analogously. We have
\begin{eqnarray}
&& \beta\big({\tt U}\cdot\beta({\tt Z},{\tt Y}),{\tt V})\big)_{pk}=\big({\tt U}\cdot\beta({\tt Z},{\tt Y})\big)_{rsp}{\tt V}_{srk}=
                  {\tt U}_{rsm}\,\beta({\tt Z},{\tt Y})_{mp}\,{\tt V}_{srk}\nonumber\\
                                    &&\qquad\qquad\qquad\qquad\qquad\qquad\qquad\qquad\quad\;={\tt U}_{rsm}\,{\tt Z}_{ijm}\,{\tt Y}_{jip}\,{\tt V}_{srk}\nonumber \\
&& \beta\big({\tt Y},{\tt Z}\cdot\beta({\tt U},{\tt V})\big)_{pk}={\tt Y}_{jip}\cdot\big({\tt Z}\cdot\beta({\tt U},{\tt V})\big)_{ijk}=
                  {\tt Y}_{jip}\,{\tt Z}_{ijm}\,\beta({\tt U},{\tt V})_{mk}\nonumber\\
                                    &&\qquad\qquad\qquad\qquad\qquad\qquad\qquad\qquad\quad\;={\tt Y}_{jip}\,{\tt Z}_{ijm}\,{\tt U}_{rsm}\,{\tt V}_{srk}\nonumber \\
&& \big(\beta({\tt Y},{\tt Z})\,\beta({\tt U},{\tt V})\big)_{pk}=\beta({\tt Y},{\tt Z})_{pm}\,\beta({\tt U},{\tt V})_{mk}=
                                {\tt Y}_{jip}{\tt Z}_{ijm}{\tt U}_{rsm}{\tt V}_{srk}.\nonumber
\end{eqnarray}
Thus, we have a second-kind associative ternary algebra of cubic matrices of order \( n \) with the ternary multiplication ${\tt X}\cdot {\tt Y}\cdot {\tt Z}={\tt X}\cdot\beta({\tt Y},{\tt Z})$. Let us denote this ternary associative algebra by $({\cal C}_n,\beta).$ Similarly we can construct the ternary associative algebra $({\cal C}_n,\gamma)$.

It should be noted that the bilinear form \( \beta \) can be expressed using the trace of a product of square matrices. Given a cubic matrix ${\tt X}=({\tt X}_{ijk})$, we can fix the value of one of its three indices, for example, the third index \(k\), by assigning it an integer value \(p\), where \(p\) is an integer from 1 to \(n\). We obtain the quantity with two indices $i,j$, which can be considered as a square matrix of $n$th order. Let us denote this $n$th order square matrix by ${\tt X}_{(p)}$, that is, ${\tt X}_{(p)}=({\tt X}_{ijp})$. Then $\beta({\tt X},{\tt Y})$ and $\gamma({\tt X},{\tt Y})$ are $n$th order square matrices and $(p,k)$-entry of these matrices can be expressed as follows
$$
\beta({\tt X},{\tt Y})_{pk}=\mbox{Tr}({\tt X}_{(p)}{\tt Y}_{(k)}),\;\;
                  \gamma({\tt X},{\tt Y})_{pk}=<{\tt X}_{(p)}{\tt Y}_{(k)}>,
$$
where $<A,B>=A_{rs}B_{rs}$ for two square matrices of $n$th order.

Thus, we have constructed a broad class of ternary $\omega$-Lie algebras using cubic matrices. For cubic matrices of order \( n \), there exist two distinct ternary $\omega$-Lie algebras. One of them is based on the ternary associative algebra $({\cal C}_n,\beta)$, while the other is based on the ternary associative algebra $({\cal C}_n,\gamma)$. We shall denote these ternary $\omega$-Lie algebras by the same symbols as the underlying ternary associative algebras.

Let us consider low-dimensional ternary $\omega$-Lie algebras. In the case \( n = 2 \), we have a vector space of cubic matrices of order two, which has dimension 8. By choosing as the ternary associative multiplication of these matrices the operation constructed using a bilinear form of trace type, namely the bilinear form \( \beta \) and equipping the vector space of cubic matrices of second order with the ternary $\omega$-commutator, we obtain an eight-dimensional ternary $\omega$-Lie algebra.  

Within this algebra, we can identify a subalgebra using the concept of the trace of a cubic matrix with respect to a pair of indices. Given a cubic matrix \( {\tt X}_{ijk} \) of order \( n \), we can define three traces by selecting pairs of indices from $i,j,k$ \cite{Sokolov}. Specifically, the trace over the first two indices is given by sum $\sum_i {\tt X}_{iij}$, the trace over the first and last indices is given by sum $\sum_i {\tt X}_{iji}$, and the trace over the last two indices is given by sum $\sum_i {\tt X}_{jii}$.  

A cubic matrix of order \( n \) is said to be traceless if all three of its traces vanish. Applying this definition to cubic matrices of order two, it is straightforward to verify that the subspace of traceless cubic matrices of order two is a two-dimensional subspace in ${\cal C}_2$. We denote this two-dimensional subspace by ${\mathfrak T}_2$. Moreover, it can be shown that the ternary $\omega$-commutator of three traceless cubic matrices results in a traceless matrix. Consequently, the traceless cubic matrices of order two form a subalgebra within the eight-dimensional ternary $\omega$-Lie algebra  of all cubic matrices of order two ${\cal C}_2$.

Let us choose the entries ${\tt X}_{111}$ and ${\tt X}_{222}$ of the traceless cubic matrix of order two ${\tt X}$ as the independent parameters. Then, the remaining entries can be expressed in terms of these two parameters as follows:
$$
{\tt X}_{221}={\tt X}_{212}={\tt X}_{122}=-{\tt X}_{111},\;\;{\tt X}_{112}={\tt X}_{121}={\tt X}_{211}=-{\tt X}_{222}.
$$
We arrange the entries of a cubic matrix of the second-order ${\tt X}$ in space, that is, in the vertices of the cube, as follows:
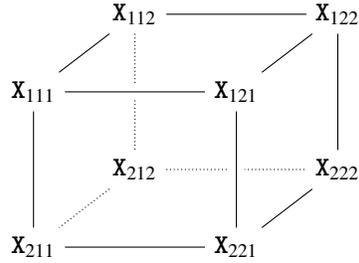
\begin{center}
\begin{tikzpicture}
  \matrix (m1) [matrix of math nodes, row sep=1.5em, column sep=1.5em]{
    & {\tt X}_{112} & & {\tt X}_{122} \\
    {\tt X}_{111} & & {\tt X}_{121} & \\
    & {\tt X}_{212} & & {\tt X}_{222} \\
    {\tt X}_{211} & & {\tt X}_{221} & \\};

  \path[-]
    (m1-1-2) edge (m1-1-4)
            edge (m1-2-1)
            edge [densely dotted] (m1-3-2)
    (m1-1-4) edge (m1-3-4)
            edge (m1-2-3)
    (m1-2-1) edge [-,line width=6pt,draw=white] (m1-2-3)
            edge (m1-2-3)
            edge (m1-4-1)
    (m1-3-2) edge [densely dotted] (m1-3-4)
            edge [densely dotted] (m1-4-1)
    (m1-4-1) edge (m1-4-3)
    (m1-3-4) edge (m1-4-3)
    (m1-2-3) edge [-,line width=6pt,draw=white] (m1-4-3)
            edge (m1-4-3);
 \end{tikzpicture}
  \end{center}
Thus, as generators of the ternary $\omega$-Lie algebra of cubic traceless matrices of the second order ${\mathfrak T}_2$, we can take two cubic matrices ${\tt F}_1=-\frac{i}{2\sqrt{2}}\,{\tt E}_1,{\tt F}_2= -\frac{i}{2\sqrt{2}}\,{\tt E}_2$, where
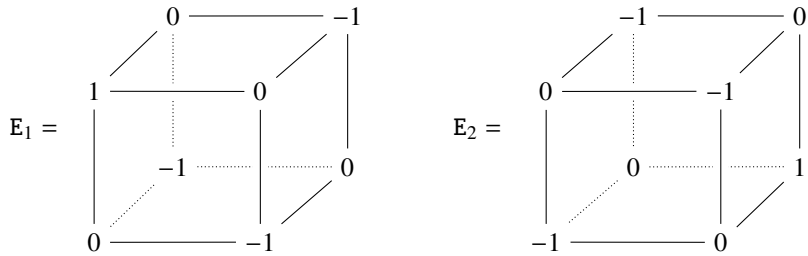
\begin{tikzpicture}
  \matrix (m1) [matrix of math nodes, row sep=1.5em, column sep=1.5em]{
    & 0 & & -1 \\
    1 & & 0 & \\
    & -1 & & 0 \\
    0 & & -1 & \\};
   \node[left=0.1em of m1] {\({\tt E}_1=\)};
  \path[-]
    (m1-1-2) edge (m1-1-4)
            edge (m1-2-1)
            edge [densely dotted] (m1-3-2)
    (m1-1-4) edge (m1-3-4)
            edge (m1-2-3)
    (m1-2-1) edge [-,line width=6pt,draw=white] (m1-2-3)
            edge (m1-2-3)
            edge (m1-4-1)
    (m1-3-2) edge [densely dotted] (m1-3-4)
            edge [densely dotted] (m1-4-1)
    (m1-4-1) edge (m1-4-3)
    (m1-3-4) edge (m1-4-3)
    (m1-2-3) edge [-,line width=6pt,draw=white] (m1-4-3)
            edge (m1-4-3);

  \matrix (m2) [matrix of math nodes, row sep=1.5em, column sep=1.5em, right=5em of m1]{
    & -1 & & 0 \\
    0 & & -1 & \\
    & 0 & & 1 \\
    -1 & & 0 & \\};
  \node[left=0.1em of m2] {\({\tt E}_2=\)};
  \path[-]
    (m2-1-2) edge (m2-1-4)
            edge (m2-2-1)
            edge [densely dotted] (m2-3-2)
    (m2-1-4) edge (m2-3-4)
            edge (m2-2-3)
    (m2-2-1) edge [-,line width=6pt,draw=white] (m2-2-3)
            edge (m2-2-3)
            edge (m2-4-1)
    (m2-3-2) edge [densely dotted] (m2-3-4)
            edge [densely dotted] (m2-4-1)
    (m2-4-1) edge (m2-4-3)
    (m2-3-4) edge (m2-4-3)
    (m2-2-3) edge [-,line width=6pt,draw=white] (m2-4-3)
            edge (m2-4-3);
\end{tikzpicture}

\noindent
Since we use $\beta$-ternary associative multiplication of cubic matrices, the ternary $\omega$-commutator of three cubic matrices ${\tt X},{\tt Y},{\tt Z}$ is the cubic matrix, whose entries can be calculated as follows
\begin{eqnarray}
[{\tt X},{\tt Y},{\tt Z}]_{ijk}&=& {\tt X}_{ijl}{\tt Y}_{nml}{\tt Z}_{mnk}+\omega\;{\tt Y}_{ijl}{\tt Z}_{nml}{\tt X}_{mnk}+
      \overline\omega\;{\tt Z}_{ijl}{\tt X}_{nml}{\tt Y}_{mnk}\nonumber\\
      &&\quad +{\tt Z}_{ijl}{\tt Y}_{nml}{\tt X}_{mnk}+
      \overline\omega\;{\tt Y}_{ijl}{\tt X}_{nml}{\tt Z}_{mnk}+\omega\;{\tt X}_{ijl}{\tt Z}_{nml}{\tt Y}_{mnk},\nonumber
\end{eqnarray}
Making use of the previous formula we find the ternary commutation relations of the ternary $\omega$-Lie algebra ${\mathfrak T}_2$
\begin{equation}
[{\tt F}_1,{\tt F}_2,{\tt F}_1]={\tt F}_2,\;\;[{\tt F}_2,{\tt F}_1,{\tt F}_2]={\tt F}_1.
\label{sec4: commutation relations 2}
\end{equation}
By comparing the obtained commutation relations \eqref{sec4: commutation relations 2} with the commutation relations \eqref{sec4: commutation relations 1}, we see that the two 2-dimensional ternary $\omega$-Lie algebras we constructed, ${\mathfrak M}_{1,2}$ and ${\mathfrak T}_2$, are isomorphic.
\section{Classification of 2-dimensional ternary Lie algebras at cube roots of unity}
In the previous section, we constructed two 2-dimensional ternary $\omega$-Lie algebras and demonstrated that they are isomorphic. The first ternary $\omega$-Lie algebra ${\mathfrak M}_{1,2}$ was constructed using vectors of the 2-dimensional complex vector space $\mathbb C^2$, while the second ${\mathfrak T}_2$ was constructed by means of traceless cubic matrices of order two. In this section, we provide a complete classification of 2-dimensional ternary $\omega$-Lie algebras. We prove a theorem that asserts that, up to isomorphism, there exist four distinct 2-dimensional ternary $\omega$-Lie algebras. The formulation of the theorem specifies the ternary commutation relations for these algebras.
\begin{theorem}
If $\cal F$ is a two-dimensional ternary $\omega$-Lie algebra, then it is isomorphic to one of the four two-dimensional ternary $\omega$-Lie algebras given by their structure constants in the \mbox{following table:}
\begin{table}[H]
\begin{tabularx}{\linewidth}{|X|X|X|X|X|}
\toprule
\textbf{{Number}} 
	& \textbf{$C^1_{121}$}	& \textbf{$C^2_{121}$} & \textbf{$C^1_{212}$}  & \textbf{$C^2_{212}$}\\
\midrule
{{I}}                  & 0		                          & 0        &     0              &     0             \\

{{II}}	                & 0			                       & 1        &     1              &     0              \\

{{III}}	            & 0		                       & 1        &   0                &     0              \\

{{IV}}	            & 1		                       & 0        &   0                &     $-$1              \\
\bottomrule
\end{tabularx}
\end{table}
\label{theorem classification}
\end{theorem}
\begin{proof}
Let $\cal F$ be a two-dimensional ternary $\omega$-Lie algebra and $e_1,e_2$ be generators for this algebra. Due to the properties of a ternary $\omega$-Lie algebra, a ternary bracket of this algebra containing three equal arguments is equal to zero. Hence, we have the following:
\begin{equation}
[e_1,e_1,e_1]=[e_2,e_2,e_2]=0.
\label{111,222 is zero}
\end{equation}
Moreover, due to the $\omega$-symmetries of a ternary bracket of a ternary $\omega$-Lie algebra, we have the following:
\begin{eqnarray}
&& [e_1,e_2,e_1] = \omega\,[e_2,e_1,e_1]=\overline{\omega}[e_1,e_1,e_2],\nonumber\\
&& [e_2,e_1,e_2] = \omega\,[e_1,e_2,e_2]=\overline{\omega}[e_2,e_2,e_1].\nonumber
\end{eqnarray}
Thus, we have two independent and possibly non-trivial ternary brackets, that is, $[e_1,e_2,e_1]$, $[e_2,e_1,e_2]$, which completely determine the structure of a two-dimensional ternary $\omega$-Lie algebra. We can expand these two ternary brackets via the structure constants \mbox{as follows:}
\begin{equation}
[e_1,e_2,e_1]=C_{121}^1\,e_1+C_{121}^2\,e_2,\qquad
      [e_2,e_1,e_2]=C_{212}^1\,e_1+C_{212}^2\,e_2.
\end{equation}
Thus, the structure of a two-dimensional ternary $\omega$-Lie algebra $\cal F$ is completely determined by four independent structure constants $C^1_{121},C^2_{121},C^1_{212},C^2_{212}$. First of all, the structure constants of a two-dimensional $\omega$-Lie algebra have the \mbox{following symmetries:}
\begin{eqnarray}
C^1_{121}=\omega\;C^1_{211}=\overline\omega\;C^1_{112},\quad C^2_{121}=\omega\;C^2_{211}=\overline\omega\;C^2_{112},\label{symmetries of 121}\\
C^1_{212}=\omega\;C^1_{122}=\overline\omega\;C^1_{221},\quad C^2_{212}=\omega\;C^2_{122}=\overline\omega\;C^2_{221}.\label{symmetries of 212}
\end{eqnarray}
From (\ref{symmetries of 121}) and (\ref{symmetries of 212}), we have the following:
\begin{equation}
C^m_{121}+C^m_{211}+C^m_{112}=0, \quad C^m_{212}+C^m_{221}+C^m_{122}=0,\qquad m=1,2.\label{sum is zero}
\end{equation}
It follows from the $GA(1,5)$- identity that the structure constants of $\cal F$ must satisfy the following equations:
\begin{equation}
\circlearrowleft (C^m_{\underline{i}\,\underline{k}\,\underline{l}}\,C^p_{m\,\underline{r}\,\underline{s}}+C^m_{\underline{i}\,\underline{r}\,\underline{k}}\,C^p_{m\,\underline{s}\,\underline{l}}+C^m_{\underline{i}\,\underline{s}\,\underline{r}}\,C^p_{m\,\underline{l}\,\underline{k}}+C^m_{\underline{i}\,\underline{l}\,\underline{s}}\,C^p_{m\,\underline{k}\,\underline{r}})=0.
\label{identity for structure constants 2}
\end{equation}
We claim that, in the case of a two-dimensional ternary $\omega$-Lie algebra $\cal F$, the $GA(1,5)$- identity does not impose additional conditions on the structure constants of $\cal F$, that is, \mbox{Equation (\ref{identity for structure constants 2})} is satisfied for any integers $(i,k,l,r,s)$, independently running the values $1,2$, by virtue of the properties (\ref{sum is zero}) of the structure constants.

Indeed, the number of ordered sequences of integers $(i,k,l,r,s)$, where each integer can be either 1 or 2, is $2^5=32$. We can discard two of these sequences $(1,1,1,1,1), (2,2,2,2,2)$, since due to property (\ref{111,222 is zero}), the left-hand side of Equation (\ref{identity for structure constants}) in this case is equal to zero. Since each sequence of integers $(i,k,l,r,s)$ determines an Equation (\ref{identity for structure constants}), formally, we have 30 equations, but not all of these equations are distinct. We will go through all possible sequences of integers $(i,k,l,r,s)$ using the following scheme. First, we will consider all sequences containing four 1s and one 2 (there will be 5 of these), then three 1s and two 2s (there will be 10 of these), then two 1s and three 2s (there will be 10 of these), and finally one 1 and four 2s (there will be 5 of these). We start with the sequence $(1,2,1,1,1)$. Through straightforward calculations, we identify that this sequence leads to the following:
\begin{equation}
C^m_{121}\,C^p_{m11}+C^m_{211}\,C^p_{m11}+C^m_{111}\,C^p_{m12}+C^m_{111}\,C^p_{m21}+C^m_{112}\,C^p_{m11}=0.
\label{equation 1211}
\end{equation}
First, note that this equation contains all ordered sequences of integers $(i,k,l,r,s)$ containing four 1s and one 2. This means that any such sequence yields Equation (\ref{equation 1211}). Second, the left-hand side of this equation is identically zero, which is easy to see if we take into account that $C^m_{111}=0$ and write the left-hand side of (\ref{equation 1211}) in the following form:
$$
(C^m_{121}+C^m_{211}+C^m_{112})\,C^p_{m11}.
$$
Because of property (\ref{sum is zero}), this expression vanishes and we conclude that (\ref{equation 1211}) is an identity and does not impose additional constraints on the structure constants.

Next, we consider the sequence $(1,2,1,1,2)$. This sequence leads to the following equation:
\begin{eqnarray}
&& C^m_{121}\,C^p_{m12}+ C^m_{211}\,C^p_{m21}+C^m_{112}\,C^p_{m12}+C^m_{121}\,C^p_{m21}+C^m_{212}\,C^p_{m11}\nonumber\\
   &&\;\;\;  +C^m_{112}\,C^p_{m21}+C^m_{122}\,C^p_{m11}+C^m_{221}\,C^p_{m11}+C^m_{211}\,C^p_{m12}+C^m_{111}\,C^p_{m22}=0.
   \label{equation 12112}
\end{eqnarray}
As in the previous case, the left-hand side of the above equation contains all ordered sequences of integers containing three 1s and two 2s (there are ten of such sequences). Since we perform cyclic permutations on the five subscripts, this means that by taking any sequence and substituting it into (\ref{identity for structure constants}), we obtain Equation (\ref{equation 12112}). Since $C^m_{111}=0$, the last term on the left-hand side of Equation (\ref{equation 12112}) is zero, and by collecting the remaining nine terms, i.e.,
$$
(C^m_{121}+C^m_{211}+C^m_{112})\,(C^p_{m12}+C^p_{m21})+(C^m_{221}+C^m_{212}+C^m_{122})\,C^p_{m11}
$$
we see that by property (\ref{sum is zero}), the left-hand side of Equation (\ref{equation 12112}) is identically zero. We still have two sets of sequences left, that is, the set of sequences containing two 1s and three 2s, and the set of sequences containing one 1 and four 2s.
By choosing sequences $(1,2,1,2,2)$ and $(2,1,2,2,2)$ as representatives of these two sets and substituting them into (\ref{identity for structure constants}), we obtain the following equations:
\begin{eqnarray}
&& C^m_{121}\,C^p_{m22}+C^m_{212}\,C^p_{m21}+C^m_{122}\,C^p_{m12}+C^m_{221}\,C^p_{m21}+C^m_{212}\,C^p_{m12}\nonumber\\
   &&\qquad\quad\quad\;\;  +C^m_{122}\,C^p_{m21}+C^m_{222}\,C^p_{m11}+C^m_{221}\,C^p_{m12}+C^m_{211}\,C^p_{m22}+C^m_{112}\,C^p_{m22}=0,\nonumber\\
&& C^m_{212}\,C^p_{m22}+C^m_{122}\,C^p_{m22}+C^m_{222}\,C^p_{m21}+C^m_{222}\,C^p_{m12}+C^m_{221}\,C^p_{m22}=0,\nonumber
\end{eqnarray}
respectively. The analysis of these two equations is similar to that given above and shows that we do not obtain any additional equation on the structure constants.

Now, we will study how the structure constants $C^m_{ikl}$ behave when we pass to another basis of two-dimensional space. Obviously, the structure constants change, transforming as a (1,3)-tensor, but the structure of a ternary $\omega$-Lie algebra $\cal F$ remains the same. Let $e_1^\prime,e_2^\prime$ be another basis of generators for the two-dimensional ternary $\omega$-Lie algebra $\cal F$, where we have the following:
\begin{equation}
e_1=\alpha^1_1\,e_1^\prime+\alpha^2_1\,e_2^\prime,\qquad
      e_2=\alpha^1_2\,e_1^\prime+\alpha^2_2\,e_2^\prime.
\end{equation}
Let us denote a transition matrix by $A$ as follows:
\begin{equation}
A=\left(
  \begin{array}{cc}
    \alpha^1_1 & \alpha^1_2 \\
    \alpha^2_1 & \alpha^2_2 \\
  \end{array}
\right).
\end{equation}
Obviously, $A$ is a regular matrix, that is, $\mbox{Det}\,A\neq 0$. Thus, $A$ belongs to the group of regular second-order complex matrices, that is,  $A\in \mbox{GL}_2(\mathbb C)$. Let us denote the structure constants of $\cal F$ in the basis $e_1^\prime, e_2^\prime$ by $C^{\prime\,\,m}_{\,ikl}$, that is,
$$
[e_1^\prime,e_2^\prime,e_1^\prime]={C}^{\prime\,\,1}_{\,121}\,e_1^\prime+{C}^{\prime\,\,2}_{\,121}\,e_2^\prime,\;\;\;
    [e_2^\prime,e_1^\prime,e_2^\prime]={C}^{\prime\,\,1}_{\,212}\,e_1^\prime+{C}^{\prime\,\,2}_{\,212}\,e_2^\prime.
$$
Through straightforward calculations, we find the following:
\begin{equation}
\left(
  \begin{array}{c}
    {C}^{\prime\,\,1}_{\,121} \\
    \\
    {C}^{\prime\,\,2}_{\,121} \\
    \\
    {C}^{\prime\,\,1}_{\,212}\\
    \\
    {C}^{\prime\,\,2}_{\,212}\\
  \end{array}
\right)=\frac{1}{(\mbox{Det}\,A)^2}\left(
  \begin{array}{ccccccc}
    \alpha^1_1\,\alpha^2_2 & \;\;\;&\alpha^1_2\,\alpha^2_2 &\;\;\; &\alpha^1_1\,\alpha^2_1 & \;\;\;&\alpha^1_2\,\alpha^2_1 \\
    \\
    \alpha^2_1\,\alpha^2_2 & & (\alpha^2_2)^2 & & (\alpha^2_1)^2 & & \alpha^2_1\,\alpha^2_2 \\
    \\
    \alpha^1_1\,\alpha^1_2 & & (\alpha^1_2)^2 & & (\alpha^1_1)^2 & & \alpha^1_1\,\alpha^1_2 \\
    \\
    \alpha^1_2\,\alpha^2_1 & &\alpha^1_2\,\alpha^2_2 & &\alpha^1_1\,\alpha^2_1 & &\alpha^1_1\,\alpha^2_2 \\
  \end{array}
\right)\;\left(
  \begin{array}{c}
  {C}^1_{121} \\
  \\
  {C}^2_{121} \\
  \\
  {C}^1_{212} \\
  \\
  {C}^2_{212} \\
  \end{array}
\right).
\label{transformation of structure constants}
\end{equation}
Obviously, Formula (\ref{transformation of structure constants}) defines the tensor representation of the Lie group $\mbox{GL}_2(\mathbb C)$ in the space of (1,3)-tensors with symmetries (\ref{symmetries of 121}) and (\ref{symmetries of 212}).

We will consider the structure constants, ordered as follows $(C^1_{121},C^2_{121},C^1_{212},C^2_{212})$, as vectors of the four-dimensional complex vector space $\mathbb C^4$. We will exclude the trivial case where all structure constants are zero. In this case, we will refer to the corresponding ternary $\omega$-Lie algebra as Abelian. In the four-dimensional complex vector space of structure constants, vectors satisfying the condition $C^1_{121}=C^2_{212}$ form the three-dimensional subspace, which will be denoted by $\cal W$. Thus, we have the following:
$$
{\cal W}=\{(C^1_{121},C^2_{121},C^1_{212},C^2_{212})\in {\mathbb C}^4: C^1_{121}=C^2_{212}\}.
$$
As follows, from (\ref{transformation of structure constants}), the subspace $\cal W$ is invariant with respect to transformations (\ref{transformation of structure constants}). Let us denote by $\cal V$ the one-dimensional subspace of $\mathbb C^4$ spanned by the vector $(1,0,0,-1)$. Then, ${\mathbb C}^4={\cal W}\oplus{\cal V}$. It follows from (\ref{transformation of structure constants}) that the vector $(1,0,0,-1)$ is an eigenvector of all transformations (\ref{transformation of structure constants}), i.e., $\cal V$ is an invariant one-dimensional subspace. Thus, the tensor representation (\ref{transformation of structure constants}) is reducible and the one-dimensional subspace $\cal V$ defines the two-dimensional ternary $\omega$-Lie algebra of our classification. The non-trivial commutation relations of this algebra will be written as follows:
\begin{equation}
[e_1,e_2,e_2]=e_1,\;\;[e_2,e_1,e_2]=-e_2.
\label{algebra 1,0,0,-1}
\end{equation}

Now, we study the structure of the three-dimensional subspace $\cal W$. To simplify the presentation, we introduce the following notations:
$$
a={C}^{\prime\,\,1}_{\,121},\;\;b={C}^{\prime\,\,2}_{\,121},\;\;c={C}^{\prime\,\,1}_{\,212},\;\;x=\alpha^1_1,\;\;y=\alpha^1_2,\;\;z=\alpha^2_1,\;\;u=\alpha^2_2.
$$
Since $A\in \mbox{GL}_2(\mathbb C)$, it holds $x\,u-y\,z\neq 0$.
First, we note that the vector $(0,1,1,0)$ of structure constants of the two-dimensional ternary $\omega$-Lie algebra ${\cal L}_2$ belongs to the subspace $\cal W$. Thus, by applying to this vector all possible transformations (\ref{transformation of structure constants}), we obtain the set of vectors lying in the subspace $\cal W$ that define the same algebra ${\cal L}_2$. Therefore, a vector $(a,b,c,a)\in {\cal W}$ defines the algebra ${\cal L}_2$ if the system of equations, i.e.,
\begin{eqnarray}
\frac{x\,z+y\,u}{(x\,u-y\,z)^2} &=& a,\nonumber\\
\frac{z^2+u^2}{(x\,u-y\,z)^2} &=& b,\label{system of equations 1}\\
\frac{x^2+y^2}{(x\,u-y\,z)^2} &=& c,\nonumber
\end{eqnarray}
obtained from (\ref{transformation of structure constants}), has at least one solution. We mean that $x,y,z,u$ are unknown (elements of a basis transformation matrix $A$), and $a,b,c$ are given numbers (structure constants ${C}^{\prime\,\,1}_{\,121},{C}^{\prime\,\,2}_{\,121},{C}^{\prime\,\,1}_{\,212}$). We will prove that if the condition $a^2=b\,c$ is satisfied, then system (\ref{system of equations 1}) has no solutions, and, consequently, the vector $(0,1,1,0)$ cannot be transformed by transformation (\ref{transformation of structure constants}) into the vector $(a,b,c,a)$. Thus, the vectors $(a,b,c,a)$ satisfying the condition $a^2=b\,c$ determine two-dimensional ternary $\omega$-Lie algebras (in fact, one algebra) that are not isomorphic to ${\cal L}_2$.

The system of Equation (\ref{system of equations 1}) can be written in matrix form, as follows:
\begin{equation}
\frac{1}{(\mbox{Det}\,A)^2}\;A\,A^{T}=\left(
  \begin{array}{cc}
    c & a \\
    a & b \\
  \end{array}
\right), \;\;\;\;A=\left(
  \begin{array}{cc}
    x & y \\
    z & u \\
  \end{array}
\right).
\end{equation}
By calculating the determinants of both sides of this matrix equation, we obtain the following:
$$
\frac{1}{(\mbox{Det}\,A)^2}=\mbox{Det}\;\left(
  \begin{array}{cc}
    c & a \\
    a & b \\
  \end{array}
\right).
$$
Thus, assuming that $a^2=b\,c$ and system (\ref{system of equations 1}) has a solution, we arrive at a contradiction since the left-hand side of the above equality cannot be zero ($\mbox{Det}\,A\neq 0$), while due to the condition $a^2=b\,c$ the right-hand side is zero. A vector of the subspace $\cal W$ satisfying $a^2=b\,c$ can be written as $(\pm\sqrt{bc},b,c,\pm\sqrt{bc})$. We can show that the vector $(0,1,0,0)\in\cal W$ can be transformed by (\ref{transformation of structure constants}) to any vector of the form $(\pm\sqrt{bc},b,c,\pm\sqrt{bc})$. Indeed, if $b\neq 0$, then the transformation matrix, i.e.,
$$
A=\left(
  \begin{array}{cc}
    \frac{1}{\sqrt{b}} & {\frac{\sqrt{c}}{\sqrt{b}}} \\
    0 & 1 \\
  \end{array}
\right)
$$
induces the transformation (\ref{transformation of structure constants}) that sends the vector $(0,1,0,0)$ to the vector $(\sqrt{bc},b,c,\sqrt{bc})$. If $b=0$, but $c\neq 0$, then the transformation matrix
$$
A=\left(
  \begin{array}{cc}
    0 & 1 \\
    \frac{1}{\sqrt{c}} & 0 \\
  \end{array}
\right),
$$
induces the transformation (\ref{transformation of structure constants}) that sends the vector $(0,1,0,0)$ to $(0,0,c,0)$. From this, we conclude that the set of vectors of type $(\pm\sqrt{bc},b,c,\pm\sqrt{bc})$ gives one more two-dimensional ternary $\omega$-Lie algebra that is not isomorphic to either algebra ${\cal L}_2$ or algebra (\ref{algebra 1,0,0,-1}). The commutation relations of this algebra can be written as follows:
\begin{equation}
[e_1,e_2,e_1]=e_2,\;\;\;[e_2,e_1,e_2]=0.
\label{algebra 0,1,0,0}
\end{equation}
The last possibility we must consider involves the vectors $(a,b,c,a)$ of the subspace $\cal W$ that satisfy condition $a^2\neq b\,c$. In a similar way to what we did above, we can prove that for any vector $(a,b,c,a)\in{\cal W}$ that satisfies the condition $a^2\neq b\,c$, the system of Equation (\ref{system of equations 1}) has solutions. This means that each vector determines the two-dimensional ternary $\omega$-Lie algebra ${\mathfrak M}_{1,2}$.
\end{proof}
In the classification table of Theorem \ref{theorem classification}, the two-dimensional ternary $\omega$-Lie algebra labeled $I$ is Abelian, and the two-dimensional ternary $\omega$-Lie algebra labeled {$II$} is ${\mathfrak M}_{1,2}$. Recall that the two-dimensional ternary $\omega$-Lie algebra ${\mathfrak M}_{1,2}$ was constructed in two ways. The first way uses vectors of the complex plane, and the second way uses traceless cubic matrices of the second order. It is easy to verify that the two-dimensional ternary $\omega$-Lie algebras labeled {$II$} (${\cal L}_2$) and {$IV$} are simple ternary $\omega$-Lie algebras and algebra {$III$} possesses the non-trivial ideal spanned by the second generator $e_2$.

\end{document}